\documentclass[a4paper]{article}

\usepackage{amscd,amsthm,amsfonts,amsmath,amssymb,indentfirst}
\usepackage[english]{babel}
\usepackage[utf8]{inputenc}
\usepackage[T1]{fontenc}

\usepackage{color}

\usepackage[hidelinks]{hyperref}

\usepackage{verbatim} 

\usepackage[showonlyrefs]{mathtools}
\mathtoolsset{showonlyrefs=true}

\providecommand{\keywords}[1]{{\small \textbf{{Keywords:}} #1}}
\providecommand{\amsclass}[1]{{\small \textbf{{Mathematics Subject Classification (2010):}} #1}}

\newcommand{\er}{\mathbb{R}}
\newcommand{\R}{\mathbb{R}}

\newcommand{\mf}[1]{\mathbf{#1}}
\newcommand{\ub}{\mathbf{u}}
\newcommand \ben {\begin{equation}}
\newcommand \een {\end{equation}}

\theoremstyle{theorem}
\newtheorem{Teorema}{Theorem}[section]
\newtheorem{Lema}[Teorema]{Lemma}

\theoremstyle{definition}
\newtheorem{Definicao}[Teorema]{Definition}

\theoremstyle{remark}

\numberwithin{equation}{section}

\title{Semitrivial vs. fully nontrivial ground states in cooperative cubic Schr\"odinger systems with $d\ge3$ equations}
\author{Sim\~ao Correia, Filipe Oliveira and Hugo Tavares}

\begin{document}

\date{\today}
\maketitle

\begin{abstract}
\noindent
In this work we consider the weakly coupled Schr\"odinger cubic system
\[
\begin{cases}
\displaystyle -\Delta u_i+\lambda_i u_i= \mu_i u_i^{3}+ u_i\sum_{j\neq i}b_{ij} u_j^2 \\
u_i\in H^1(\R^N;\er), \quad i=1,\ldots, d,
\end{cases}
\]
where $1\leq N\leq 3$,  $\lambda_i,\mu_i >0$ and $b_{ij}=b_{ji}>0$ for $i\neq j$. This system admits semitrivial solutions, that is solutions $\mf{u}=(u_1,\ldots, u_d)$ with null components. We provide optimal qualitative conditions on the parameters $\lambda_i,\mu_i$ and $b_{ij}$ under which the \emph{ground state solutions} have all components nontrivial, or, conversely, are semitrivial.

\noindent
This question had been clarified only in the $d=2$ equations case. For $d\geq 3$ equations, prior to the present paper, only very restrictive results were known, namely when the above system was a small perturbation of the super-symmetrical case $\lambda_i\equiv \lambda$ and $b_{ij}\equiv b$. We treat the general case, uncovering in particular a much more complex and richer structure with respect to the $d=2$ case.
\end{abstract}

\medbreak

\noindent \keywords{Cooperative systems, cubic Schr\"odinger systems, existence and nonexistence results, gradient elliptic systems, ground states, semitrivial and fully nontrivial solutions}

\noindent   \amsclass{35J47, 35J50 (Primary); 35B08, 35B09, 35Q55}

\section{Introduction}
We are interested in the elliptic system of $d$ equations
 \begin{equation}
 \label{sistema}
\begin{cases}
\displaystyle -\Delta u_i+\lambda_i u_i= \mu_i u_i^{3}+ u_i\sum_{j\neq i}b_{ij} u_j^2 \\
u_i\in H^1(\R^N;\er), \quad i=1,\ldots, d,
\end{cases}
\end{equation}
in $\R^N$, $1\leq N\leq 3$, with $\lambda_i,\mu_i >0$ for every $i=1,\ldots, d$ and $b_{ij}=b_{ji}>0$ for $i\neq j$. This system arises naturally when looking for standing wave solutions $\Psi_i(x,t)=e^{-i\lambda_i t}u_i(x)$ of the cubic nonlinear Schr\"odinger system
\[
i\partial_t \Psi_i-\Delta \Psi_i= \mu_i|\Psi_i|^2\Psi_i+ \Psi_i \sum_{j\neq i} b_{ij} |\Psi_j|^2,\quad i=1,\dots, d.
\]
The parameters $\mu_i$ represent self-interactions within the same component, while $b_{ij}$ ($i\neq j$) express the strength and the type of interaction between different components $i$ and $j$. When $b_{ij}>0$, this interaction is said of cooperative type, modeling phenomena appearing in nonlinear optics (see \cite{Sirakov} and the physical references therein). On the other hand, a negative coefficient $b_{ij}$ denotes competition, a feature arising, for instance, when modeling the Bose-Einstein condensation (see for instance \cite{Timmermans}).

\medskip

The assumption $b_{ij}=b_{ji}$, which translates the fact that the interactions between components are symmetric, implies that the system is of gradient type, and solutions of \eqref{sistema} correspond then to the critical points of the $C^2$--action functional $I_d:(H^1(\R^N))^d\to \R$ defined by
\[
I_d(\ub)=I_d(u_1,\ldots, u_d):=\frac 12 \sum_{i=1}^d \|u_i\|_{\lambda_i}^2-\frac 1{4} \sum_{i=1}^d  \mu_i |u_i|^{4}_{4}-\frac{1}{2}\mathop{\sum_{i,j=1}^d}_{i<j}b_{ij}|u_iu_j|_2^2,
\] 
where 
\[
\|v\|_{\lambda_i}^2:=\int |\nabla v|^2+\lambda_i\int v^2,
\]
and $|\cdot|_p$ stands for the usual $L^p$-norm, $1\leq p\leq \infty$.
\\
Among all eventual solutions of \eqref{sistema}, some of the most relevant, both from a physical and mathematical point of view\footnote{The reader may refer for instance to the excellent introduction in \cite{Sirakov} for more details}, are the so called \emph{ground states} (or \emph{least action} solutions), that is, solutions achieving the \emph{ground state level}
\[
c:=\inf\{ I_d(\ub):\ \ub\neq \mf{0},\ I_d'(\ub)=0\}.
\]
In this work we will denote the set of all minimizers of this problem by $G$, that is,
\[
G:=\{\ub:\ \ub\neq \mf{0},\ I_d'(\ub)=0,\ I_d(\ub)=c \}.
\]
Under our assumptions, for all values of the parameters $\lambda_i$, $\mu_i$ and $b_{ij}$, the existence of ground state solutions is not an issue. In fact, there exists a non-negative radially decreasing ground state $\mathbf{u}\in \mathcal{N}_d$, and it is quite classical to check that the set of minimizers of the problem 
\[
\inf \{I_d(\mathbf{u})\,:\,\mathbf{u}\in\mathcal{N}_d\}
\]
is exactly $G$, where $\mathcal{N}_d$ is the Nehari manifold
\[
\mathcal{N}_d:=\{\mathbf{u}\in (H^1(\er^N))^d\,:\,\mathbf{u}\neq 0, \ I_d'(\mathbf{u})[\ub]=0\}. 
\]
For complete proofs of these facts, we refer for instance to \cite[Proposition 10]{Correia1} (see also \cite[Theorem 2.1]{MaiaMontefuscoPellacci}, or \cite[Theorem 0.1]{BartschWang} for $N\geq 2$ and \cite[Proposition 2.1]{OliveiraTavares} for $N\geq 1$). Observe that, by the maximum principle, all ground state solutions have signed components. Moreover, note that
$$\mathbf{u}\in\mathcal{N}_d\Rightarrow \tau_d(\mathbf{u}):=\sum_{i=1}^d \|u_i\|^2_{\lambda_i}-\Big( \sum_{i=1}^d \mu_i|u_i|_{4}^{4}+2\sum_{i<j}b_{ij}|u_iu_j|_2^2  \Big)=0,$$ hence, for every  $\mathbf{u}\in \mathcal{N}_d$,
\[
I_d(\mathbf{u}) =\frac 14\sum_{i=1}^d \|u_i\|^2_{\lambda_i} 
			   =\frac 14\Big( \sum_{i=1}^d \mu_i|u_i|_{4}^{4}+2\sum_{i<j}b_{ij}|u_iu_j|_2^2  \Big).
\]

\medskip

An interesting and much more challenging question than the problem of the existence of ground states is whether the system (\ref{sistema}) admits solutions $\mf{u}=(u_1,\ldots, u_d)$ such that $u_i\not \equiv 0$ for every $i$. Such solutions will be called \emph{fully nontrivial solutions}. On the other hand, solutions with at least one vanishing component will be said \emph{semitrivial}. We will be particularly interested in the question of whether the ground states are fully nontrivial or not, and for that reason we introduce the set
\[
G^*=\{\mf{u}: \ \mf{u} \text{ is a ground state solutions of \eqref{sistema} with } u_i\not\equiv 0 \ \forall i
\}.
\]
Up to now, this question had only been answered in some particular cases; our aim is to fully describe how, for the general system \eqref{sistema}, the parameters $\lambda_i,\mu_i$ and $b_{ij}$ influence ground states to be either fully nontrivial or semitrivial.

\bigskip

Before stating our main results, let us describe accurately what was known before the present paper. In the 2--equation case
\begin{equation}\label{sistema2}
\begin{cases}
 -\Delta u_1+\lambda_1 u_1= \mu_1 u_1^{3}+ b_{12} u_1u_2^2 \\
  -\Delta u_2+\lambda_2 u_2= \mu_2 u_2^{3}+ b_{12}u_2u_1^2,\\
\end{cases}
\end{equation}
after the pioneer works by Ambrosetti and Colorado \cite{AmbrosettiColorado2,AmbrosettiColorado}, Lin and Wei \cite{LinWei}, Sirakov \cite{Sirakov} and Maia, Montefusco and Pellacci \cite{MaiaMontefuscoPellacci}, this question was completely solved in a recent publication by Mandel (\cite{Mandel}): in \cite[Theorem 1]{Mandel}, it is proved that  there exists $\bar b:=\bar b(\lambda_2/\lambda_1)>0$ such that for $0<b_{12}<\bar b$ all ground states are semitrivial, while for $b_{12}>\bar b$ all ground states are fully nontrivial. Moreover, a precise characterization of the threshold $\bar b$ is provided in terms of a minimization problem that involves the unique (up to translation) positive solution of the single equation case.
\bigskip

The case of three or more equations, as will come out from our results, is much richer and more complex to study. Up to now, only results when the parameters coincide (or are very close) appear in the literature. More precisely, in Liu and Wang \cite[Corollary 2.3]{LiuWang}, for 
\[
\lambda_1=\ldots=\lambda_d \qquad \text{ and } b_{ij}\equiv b>d(d-1)\max_i\{\mu_i\}-\frac{d-1}{d}\sum_j \mu_j,
\] 
it is shown that $G=G^*$, that is, all ground states are fully nontrivial. This estimate was slightly improved to a technical condition that includes the case $b_{ij}\equiv b>\max\{\mu_1\ldots, \mu_d\}$ by Liu, Liu and Chang (see Theorem 1.7 and Remarks 1.2-(b) and 1.8-(a) in \cite{LiuLiuChang}). We also refer the reader to \cite{Chang} for the particular case of $d=3$, an to Theorem 1.6 and Remark 3 of \cite{Soave}. 

Recently, it was proven in \cite[Theorem 1]{Correia3} that, when $\lambda_1=\dots=\lambda_d$, these questions may be reduced to a maximization problem in $\R^d$ and to solving a linear system. This reduction allowed the construction of examples (see Section 6 in \cite{Correia3}) which gave evidence, for the first time, of the increase in complexity when one passes from $d=2$ to $d\ge3$ equations.

\bigskip

The general case, however, was never tackled, and this is our main contribution to this subject. We will state qualitatively what kind of combinations on the parameters give rise either to semitrivial or to fully nontrivial ground states. In particular, it will be evident from our analysis that the different families of parameters play distinct roles: while the choice of the $\mu_i$ coefficients can be somehow arbitrary, only some combinations between different $\lambda_i$, and also between different $b_{ij}$ allow fully nontrivial ground states to arise.

\medskip

Let us now describe our main results, which we divided into two different groups: existence results on one hand and nonexistence results on the other. Even though it would have been possible to present fewer theorems with more elaborated statements, we wish to clearly highlight the different roles played by the $\lambda_i$ and the $b_{ij}$ parameters on the existence of fully nontrivial ground states, and for this reason we approach these two families separately. From now on, we fix $\mu_1,\ldots, \mu_d>0$.\\
Our first class of results concerns the parameters $\lambda_i$. Before stating the first theorem, let us introduce the notion of ``$\alpha$-admissibility'':

\begin{Definicao}
 Let $\alpha> 1$ and $k\geq 2$. We say that a vector $\mf{a}=(a_1,\dots,a_k)\in\er^k$ is \it{$\alpha$-admissible} $({\mf{a}}\in\mathcal{A}_{\alpha})$ if 
 \[
 \max_{1\le i \le k} a_i<\alpha  \min_{1\le i \le k} a_i.
 \]
 \end{Definicao}
 
 \begin{Teorema}[Nonexistence Result I]\label{Tnaoexistencia1}
Let $d\ge 3$, $0<\lambda_1\le \lambda_2 \le \cdots \le \lambda_d$ and $b_{ij}\equiv b>0$.\\
There exists a constant $\Lambda=\Lambda(\lambda_1/\lambda_2)$ such that, if $\lambda_2 \Lambda \leq \lambda_i$ for some $i\geq 3$ and $b>\max\{\mu_1,\ldots, \mu_d\}$, then every ground state solution $\mathbf{u}$ of (\ref{sistema}) is such that $u_i\equiv \ldots u_d\equiv 0$. In particular, we have that if $(\lambda_2,\ldots, \lambda_d)\not\in \mathcal{A}_\Lambda$, then $G^*=\emptyset$.
\end{Teorema}

The counter part of this nonexistence result is the following:

\begin{Teorema}[Existence Result I]\label{Tadmissibilidade2}
Let $d\geq 3$, $0<\lambda_1\le \lambda_2\le \cdots \le \lambda_d$ and $b_{ij}\equiv b$. Then, setting $\omega=\lambda_2/\lambda_1$, $\rho(d)=(d-2)/(d-1)$ and
\begin{equation}
\alpha=\alpha(\lambda_1/\lambda_2,d,N):=\left(1- \frac{\rho(d)-\rho(d-1)}{\sqrt{2\omega^2\frac{(\rho(d-1)+\omega)^2+\omega^2}{(\rho(d-1)+2\omega)^2}} + \rho(d)}\right)^{-\frac{2}{4-N}},
\end{equation}
if 
$$
(\lambda_2,\dots,\lambda_d)\in \mathcal{A}_{\alpha},
$$ 
there exists a constant $B=B(\lambda_i,\mu_i)>0$, such that, for $b>B$, $G=G^*$.
\end{Teorema}

The proof of this last theorem uses a classification result which we consider to be of independent interest (see Theorem \ref{Tlambdaigual} ahead): it states that if two of the $\lambda_i$ coefficients are equal, then, in any ground state, the corresponding components are proportional. Observe that, although we do not exhibit the exact threshold constants, Theorems \ref{Tnaoexistencia1} and \ref{Tadmissibilidade2} are complementary, and in the qualitative sense optimal for large $b_{ij}$. For small $b_{ij}$ we typically have semitrivial ground states, and we will come to this by the end of this introduction (see Theorem \ref{Tnaoexistencia2} ahead and the paragraph afterwards).\\
\medbreak

 It is also possible to obtain a more refined $\alpha$--admissibility condition in the following particular case:

\begin{Teorema}[Existence Result II]
 \label{Tadmissibilidade1}

Let $d\geq 3$, $0<\lambda_1, \cdots,\lambda_d$ and $b_{ij}\equiv b>0$, with
\[ 
(\lambda_1,\cdots,\lambda_d)\in\mathcal{A}_{1+\frac1{d-2}}.
\]
Then there exists a constant $B=B(\lambda_i,\mu_i)>0$, such that, for $b>B$, $G=G^*$.
\end{Teorema}
We observe that the last two results do not correspond to a ``small'' perturbation of the $\lambda_1=\ldots=\lambda_d$ case, as in \cite[Remark 2.2]{LiuWang}. On the other hand, Corollary 3 of \cite{Correia3} is similar in nature to our Theorem \ref{Tadmissibilidade1}, but it is more restrictive on the admissibility condition.

 The proofs of Theorems \ref{Tadmissibilidade2} and \ref{Tadmissibilidade1} use different arguments: while the one of Theorem \ref{Tadmissibilidade2} is based on a technical argument analogous to the one of \cite[Corollary 3]{Correia3}, the proof of Theorem \ref{Tadmissibilidade1} relies on an induction argument similar to the one used in \cite{OliveiraTavares}.

\medskip

Turning now to the results regarding the coefficients $b_{ij}$, we have the following nonexistence result which states that if some of the $b_{ij}$ are too far apart from the others, then all ground states are necessarily semitrivial.
 
 \begin{Teorema}[Nonexistence Result II]\label{Tbetanaoadmissivel2}
Take a subset $P\subset \{1,\dots, d\}$ with $\#P\ge 2$. There exists a constant $B=B((\lambda_i, \mu_i)_{1\le i\le d}, (b_{ij})_{1\le i,j\le d, (i,j)\notin P^2})$ such that, if $\min_{(i,j)\in P^2} b_{ij}>B$, then any ground state $\mf{u}$ of system \eqref{sistema} satisfies $u_{i_0}\equiv 0$ for every $i_0\not\in P$. In particular, $G^*=\emptyset$.
\end{Teorema}
 As a counter part, we will prove the following complementary existence result:
 \begin{Teorema}[Existence Result III]
\label{existbetageral}
 Consider the system \eqref{sistema} with $d\geq 3$, $\lambda=\lambda_1=\cdots=\lambda_d$ and $b_{ij}=b_{ji}>0$. Suppose that 
 \[
 \alpha:=\min_{i} \left(\min_j b_{ij}- \mu_i\right)>0
 \] 
 and 
 \[
  \max_{\substack{1\leq i\leq d\\ k\neq j}}|b_{ij}-b_{ik}|<\frac{\alpha}{d-2}.
  \]
 Then $G=G^*$.
\end{Teorema}
Observe that the bound on the distance between the parameters $b_{ij}$ increases with the distance between the smallest $b_{ij}$ and the largest $\mu_i$ is also increased. Once again, this result does not constitute a perturbation of the super-symmetric case $b_{ij}\equiv b$.

\bigskip

\bigskip

In conclusion, let us point out that, in a sense, the known results for the 2-equation case extend to $3$ or more equations in the following {\it a priori} non obvious way:
\begin{itemize}
\item In terms of the $\lambda_i$ coefficients, in order to get fully nontrivial ground states, the two lowest coefficients (say $\lambda_1$ and $\lambda_2$, $\lambda_1\leq \lambda_2$) can be chosen arbitrarily, while the remaining ($\lambda_i$, $i\geq 3$) cannot lie too far apart from $\lambda_2$, by a quantity depending on $\lambda_1/\lambda_2$.\\
Heuristically, larger $\lambda_i$ make the action larger. Therefore the conclusion is that the components $u_1$ and $u_2$ of the ground states (which are the components associated to $\lambda_1$ and $\lambda_2$) can always ``survive'' (i.e. $u_1,u_2\not\equiv 0$) when these parameters increase,  while the non-nullity of the remaining components $u_i,\ i\geq 3$, will depend on how large the respective $\lambda_i/\lambda_2$ quotients are.
\item In terms of the $b_{ij}$ coefficients:  Theorem \ref{Tbetanaoadmissivel2} states in particular that if for some $i_0$ all $b_{ij}$, with $i,j\neq i_0$, are sufficiently large compared with $b_{i_0j},\ j=1,\dots,d$, then $G^*=\emptyset$.\\Heuristically, since a larger $b_{ij}$ implies a lower action, as soon as (let's say) $b_{12}$ becomes much larger than the remaining $b_{ij}$, the lowest action only takes in consideration the first two components, and all of the remaining components of the ground states become null.\\
Hence, Theorem \ref{existbetageral} is optimal in the sense that, as they increase, all the $b_{ij}$ coefficients must remain grouped in order for a fully nontrivial ground state to exist. This is in contrast with the $\lambda_i$ coefficients: the smaller one does not need to stay packed with the remaining ones in order for fully nontrivial ground states to exist.
\end{itemize}
Also, notice that our results show how much, in general, the elements of a certain class of parameters can be apart from each other depending on the remaining classes, and that there exists fully nontrivial ground states even in situations that do not correspond to a perturbation of the super-symmetrical case.
\\
\\
Finally, note that all the existence results stated in this paper ask for large values of the parameters $b_{ij}$. This is reasonable  in the sense that, for small values of these parameters  with respect to $\mu_1,\ldots, \mu_d$, the ground states of \eqref{sistema} are semitrivial. This was proved in the 2-equation case (see \cite{MaiaMontefuscoPellacci,Mandel}), and known in the general case for $b_{ij}\equiv b$ sufficiently small (see \cite[Proposition 6]{Correia3}). We improve this result, by once again giving an explicit bound on the size of $b_{ij}\equiv b$:
\begin{Teorema}[Nonexistence Result III]\label{Tnaoexistencia2} Let $b_{ij}\equiv b>0$. Assume, without loss of generality, that $\mu_1\leq\mu_2\leq\dots\leq\mu_d$.\\
Then, if $$b<2^{1-\frac d2}\sqrt{{\mu_1}\mu_d},$$ $G^*=\emptyset$.
\end{Teorema}

Recall that, generally speaking, when $b_{ij}$ are neither small nor large with respect to all $\mu_i$, then there are no positive solutions at all: see \cite[Theorem 1-(ii)]{BartschWang}, \cite[Theorem 1- (ii)]{Sirakov} and \cite[Proposition 1.1]{Soave}.
\bigskip

The rest of this paper is organized as follows: the proofs of the existence results --- Theorems \ref{Tadmissibilidade2}, \ref{Tadmissibilidade1} and \ref{existbetageral} ---  are contained in the next section. The proofs of the nonexistence results --- Theorems \ref{Tnaoexistencia1}, \ref{Tbetanaoadmissivel2} and \ref{Tnaoexistencia2} --- are included in Section \ref{sec:Nonexistence}. In Section \ref{sec:OpenProblems} we state a few open problems and we close the paper with an Appendix where we have shown a monotonicity  result concerning the ground state's energy levels with respect to the parameters.

\section{Existence of fully nontrivial ground states}\label{sec:Existence}
In this section we prove the existence results. For convenience, we will begin with Theorem \ref{Tadmissibilidade1} and treat afterwards Theorems \ref{Tadmissibilidade2} and \ref{existbetageral}.

\begin{proof}[Proof of Theorem \ref{Tadmissibilidade1}]
We will proceed by mathematical induction on the number of equations. As mentioned in the introduction, it is well-known that the result holds true for $d=2$ equations. Indeed, in this case, for all $\lambda_1,\lambda_2>0$, the system (\ref{sistema}) admits a fully nontrivial ground state for $b$ large enough.

\bigskip

We now consider an admissible $\lambda=(\lambda_1,\dots,\lambda_d)\in\mathcal{A}_{\alpha(d)}$, with
\[
\alpha(d)=
1+\frac 1{d-2}.
\]
and, given $I\subsetneq\{1,2,\ldots, d\}$, we denote by $c(I)$ the ground state level of the system

\[
-\Delta u_i+\lambda_i u_i= \mu_i u_i^{3}+bu_i\sum_{j\in I,j\neq i}u_j^2,\qquad i\in I.
\] 

Notice that if $I\subset\{1,\dots,d\}$, then $(\lambda_i\,:\,i\in I)\in \mathcal{A}_{\alpha(\#I)}$.
Hence, following the ideas in \cite{OliveiraTavares}, we assume, by induction hypothesis, that there exists a ground state level  $c(I)$ with $\#I=d-1$ and for all $J$ with $\#J<d-1$, $c(I)<c(J)$. Without loss of generality, we assume that
\[
c^{sem}:=c({\{1,\ldots, d-1\}})=\min \{ c(I):\ \#I=d-1\},
\]
where $c^{sem}$ is achieved by the fully nontrivial ground state $(u_1,\ldots, u_{d-1})\in \mathcal{N}_{d-1}$, solution of
\begin{equation}
 \label{d-1}
-\Delta u_i+\lambda_i u_i= \mu_i u_i^{3}+bu_i \mathop{\sum_{j=1}^{d-1}}_{j\neq i}u_j^2,\qquad i=1,\ldots, d-1.
\end{equation}
Noticing that $I_d(u_1,\ldots,u_{d-1},0)=I_{d-1}(u_1,\ldots, u_{d-1})$, we will prove our result by exhibiting $(U_1,\ldots ,U_d)\in\mathcal{N}_d$, $U_i\neq 0$, such that $I_d(U_1,\ldots, U_d)<I_d(u_1,\ldots, u_{d-1},0)=c^{sem}$, which guarantees that the energy level of $(U_1,\ldots, U_d)$ is inferior to the energy level of any solution of $(\ref{sistema})$ with trivial components.\\
For a fixed $w\in H^1(\er^N)$, $w\neq 0$, and $\theta>0$, we choose $t>0$ such that
\begin{equation}
 \label{condicaoNehari}
(U_1,\ldots ,U_d)=(tu_1,\ldots, tu_{d-1},t\theta w)\in\mathcal{N}_d.
\end{equation}
A straightforward computation leads to
\begin{equation}
 \label{valordet}
t^{2}=\frac{1+\theta^2C_1}{\displaystyle 1+\mu_d \theta^{4}C_2+2b\sum_{i=1}^{d-1}\theta^2 D_i},
\end{equation}
where
$$C_1=\frac{\|w\|_{\lambda_d}^2}{\displaystyle \sum_{i=1}^{d-1}\|u_i\|_{\lambda_i}^2}, \quad C_2=\frac{|w|_{4}^{4}}{\displaystyle \sum_{i=1}^{d-1} \|u_i\|_{\lambda_i}^2}\textrm{ and }D_i=\frac{|u_iw |_{2}^2}{\displaystyle \sum_{i=1}^{d-1} \|u_i\|_{\lambda_i}^2}.$$
Now, since $(tu_1,\ldots, tu_{d-1},t\theta w)\in\mathcal{N}_d$, 
\begin{align*}
I_d(tu_1,\ldots,tu_{d-1},t\theta w)&=\frac{1}{4}\Big(\sum_{i=1}^{d-1} \|t u_i\|_{\lambda_i}^2   +\theta^2\|t\theta w\|_{\lambda_d}^2\Big)\\
&=\frac{t^2}{4}\Big(1+C_1\theta^2\Big)\sum_{i=1}^{d-1} \|u_i\|_{\lambda_i}^2.
\end{align*}
and 
$$I_d(tu_1,\ldots,tu_{d-1},t\theta w)<\frac 14\sum_{i=1}^{d-1}\|u_i\|_{\lambda_i}^2=I_d(u_1,\dots,u_{d-1},0)$$
if and only if
\begin{equation}
 \label{condicaofinal}
\frac{(1+\theta^2C_1)^2-1-\mu_d \theta^{4}C_2}{\theta^2}<2b\sum_{i=1}^{d-1}D_i.
\end{equation}
By taking the limit $\theta\to 0^+$ we conclude that we only need to exhibit $w$ such that  $C_1<b\sum_{i=1}^{d-1}D_i$, that is
\begin{equation}
 \label{dois}
 \|w\|_{\lambda_d}^2<b\sum_{i=1}^{d-1}|u_iw|_2^2.
\end{equation}
This is straightforward if there exists $1\leq i_0\leq d-1$ such that $\lambda_d\leq \lambda_{i_0}$.\\
Indeed, in this case, by multiplying the equation
$$
-\Delta u_{i_0}+\lambda_{i_0}u_{i_0}=\mu_{i_0}u_{i_0}^3+bu_{i_0}\sum_{{\substack{i=1\\  i\neq i_{0}}}}^{d-1}u_{i}^2
$$
by $u_{i_0}$ and integrating, we obtain, for $b>\max\{\mu_i\,:\,1\leq i\leq d\},$
\begin{equation}
 \label{escolhafacil}
\|u_{i_0}\|_{\lambda_{d}}^2 \leq \|u_{i_0}\|_{\lambda_{i_0}}^2= \mu_{i_0}|u_{i_0}|_4^4+b\sum_{{\substack{i=1\\  i\neq i_{0}}}}^{d-1}|u_{i}u_{i_0}|_2^2<b\sum_{i=1}^{d-1}|u_{i}u_{i_0}|_2^2
\end{equation}
and we only need to choose $w=u_{i_0}$.
Hence, in the rest of the this proof, we may assume that $\lambda_d>\lambda_i$ for all $1\leq i\leq d-1$.\\
Without loss of generality, we may also assume that 
\begin{equation}\label{quatro}
|u_1|_4\geq |u_i|_4\textrm{ for all }1\leq j\leq d-1. 
\end{equation}
We then choose $w=u_1$ and, since
\begin{equation}
\label{tres} 
\|u_{1}\|_{\lambda_{1}}^2= \mu_{1}|u_{1}|_4^4+b\sum_{i=2}^{d-1}|u_{i}u_{1}|_2^2,
\end{equation}
the condition \eqref{dois} is equivalent to
\begin{equation}
 \lambda_d-\lambda_1<(b-\mu_1)\frac{|u_1|_4^4}{|u_1|_2^2}.
\end{equation}
By \eqref{tres},
$$\lambda_1<\mu_1\frac{|u_1|_4^4}{|u_1|_2^2}+b\sum_{i=2}^{d-1}\frac{|u_1u_i|_2^2}{|u_1|_2^2}<\mu_1\frac{|u_1|_4^4}{|u_1|_2^2}+\frac b2\sum_{i=2}^{d-1}\frac{|u_1|_4^4+|u_i|_4^4}{|u_1|_2^2}$$
\begin{equation}
 \label{seis}
<((d-2)b+\mu_1)\frac{|u_1|_4^4}{|u_1|_2^2},
\end{equation}
where we have used the hypothesis \eqref{quatro}.\\
Hence, for \eqref{dois} to hold, it is sufficient that $\displaystyle \lambda_d-\lambda_1<\frac{b-\mu_1}{(d-2)b+\mu_1}{\lambda_1}$.\\
By taking the limit $b\to +\infty$, this condition holds for large $b$ if
$\lambda_d<(1+\frac 1{d-2})\lambda_1$, which is true if $(\lambda_1,\dots,\lambda_d)\in\mathcal{A}_{1+\frac 1{d-2}}$. \qedhere
\end{proof}

Before we address Theorem \ref{Tadmissibilidade2}, we need the following classification result, which we think to be of independent interest. It corresponds to a substantial improvement of \cite[Theorem 1]{Correia3}, which holds in the weaker case $\lambda_i= b$ for \emph{all} $i=1,\ldots, d$.

\begin{Teorema}\label{Tlambdaigual}
Let $d\ge3$, $b_{ij}\equiv b>0$ and $0<\lambda_1,\dots, \lambda_d$. Suppose that, for some $k\in\{2,\dots, d\}$, $\lambda_1=\cdots=\lambda_k\equiv\lambda$. Consider the function $f:\er^k\mapsto \er$ defined by
$$
f(x_1,\cdots,x_k)=\mathop{\sum_{i,j=1}^k}_{i\neq j} bx_i^2x_j^2 + \sum_{i=1}^k \mu_i x_i^4,
$$
set $f_{max}=\displaystyle\max_{|X|=1}f(X)$ and define the set $\mathcal{X}$ in the following way:
\begin{enumerate}
\item if $\max\{\mu_1,\ldots, \mu_k\}>b$, setting $e_i$ to be the $i$-th vector of the canonical basis of $\er^k$,
$$
\mathcal{X}=\left\{ \pm e_i: \textrm{  $i$ is such that } \mu_i=\max_{j=1,\dots,k}\mu_j \right\};
$$
\item if $\max\{ \mu_1,\ldots, \mu_k\}<b$, then $f_{max}<b$ and
$$
\mathcal{X}=\left\{ X=(x_1,\dots,x_k)\in\er^k: x_i=\pm\left(\frac{b-f_{max}}{b-\mu_i}\right)^{1/2} \right\};
$$
\item if $\max \{ \mu_1,\ldots, \mu_k\}=b$,
$$
\mathcal{X}=\left\{ X=(x_1,\dots,x_k)\in\er^k:\ |X|=1 \textrm{ and } x_i=0\ \forall i \textrm{ such that }\mu_i<b\right\}.
$$
\end{enumerate}

%Then there exists a constant $\mu=\mu(\mu_1,\dots,\mu_k,b)$ 
Then, denoting by $G_{d-k+1}$ the set of ground states for the system
$$
\begin{cases}
\displaystyle -\Delta u + \lambda u= \mu u^3 + b u\sum_{j>k}u_j^2\\[10pt]
\displaystyle -\Delta u_i + \lambda_i u_i = \mu_iu_i^3 + bu^2u_i+ bu_i \mathop{\sum_{j>k}}_{j\neq i}u_j^2 ,\ i=k+1,\dots, d
\end{cases}
$$
with $\mu=f_{max}$, one has
$$
G=\{(Xu,u_{k+1},\dots, u_d)\in (H^1(\er^N))^d: (u,u_{k+1},\dots, u_d)\in G_{d-k+1}, X\in \mathcal{X}\}.
$$
\end{Teorema}

Observe that, in the previous statement, the order of the $\lambda_i$ is not important, since one may switch equations in the system. Hence Theorem \ref{Tlambdaigual} states that, if $k$ components have the same $\lambda$, then any ground state has those components proportional to each other.

\begin{proof}[Proof of Theorem \ref{Tlambdaigual}] We proceed in two steps.

\noindent\textit{Step 1. Characterization of ground states.}
Take $\mathbf{u}=(u_1,\cdots,u_d)\in G$, and let us show that $\mathbf{u}=(Xu,u_{k+1},\dots, u_d)$, where $(u,u_{k+1},\dots, u_d)\in G_{d-k+1}$ and $X\in \mathcal{X}$. Define
\[
u(x)=\left(\sum_{i=1}^k u_i^2(x)\right)^{1/2}.
\]
If $u=0$, there is nothing left to prove. Otherwise, let $\mathcal{X}$ be the (nonempty) set of solutions to the maximization problem
\begin{equation}\label{maximizacao}
f(X_0)=f_{max}=\max_{|X|=1} f(X), \quad |X_0|=1.
\end{equation}

Take $X_0\in\mathcal{X}$ and $W=(w_1,\dots,w_d)=(X_0u,u_{k+1},\dots,u_d)$. Let us show that $W$ is also a ground state solution. We have
\begin{align*}
&\sum_{i=1}^d \mu_i |u_i|_{4}^{4}+\sum_{j\neq i} b|u_iu_j |_2^2=  \int f(|u_1|,\cdots,|u_{d-1}|) +\sum_{i>k}\int \mu_i u_i^4 \\
 & + \sum_{i,j>k,i\neq j} \int bu_i^2u_j^2 +\sum_{i>k}\sum_{j=1}^{k} \int 2bu_i^2u_j^2  \\
= &\sum_{i>k}\int \mu_i u_i^4 + b\sum_{i,j>k,i\neq j} \int u_i^2u_j^2 +2b\sum_{i>k}\int u_i^2u^2 +\int f\left(\frac{|u_1|}{u}, \cdots, \frac{|u_{d-1}|}{u}\right)u^4 \\ 
\le &  \sum_{i>k}\int \mu_i u_i^4 + b\sum_{i,j>k, i\neq j} \int u_i^2u_j^2 +2b\sum_{i>k} \int u_i^2|X_0|^2u^2  + \int  f(X_0)u^4 \\ =\ &\sum_{i=1}^d \mu_i |w_i|_{4}^{4}+\sum_{j\neq i} b|w_iw_j |_2^2
\end{align*}
 Furthermore, since $|\nabla \mathbf{u}|^2_2\leq \sum_{i=1}^k |\nabla u_i|_2^2$,%using the diamagnetic inequality $|\nabla |v||_2^2\le |\nabla v|_2^2$, 
\[
 \sum_{i=1}^d \|w_i\|_{\lambda_i}^2=\sum_{i=k+1}\|u_i\|_{\lambda_i}^2 + \sum_{i=1}^{k} \|(X_0)_iu\|^2_{\lambda} =\sum_{i=k+1}^d\|u_i\|_{\lambda_i}^2 + \|u\|^2_{\lambda} \le\sum_{i=1}^k\|u_i\|_{\lambda_i}^2.
 \]
Take $t>0$ such that $tW\in\mathcal{N}_d$. Then the above inequalities show that
\[
t^2=\frac{\sum_{i=1}^k \|w_i\|_{\lambda_i}^2}{\sum_{i=1}^d \mu_i |w_i|_4^4+\sum_{j\neq i} b|w_i w_j|^2_2}\leq \frac{\sum_{i=1}^k \|u_i\|_{\lambda_i}^2}{\sum_{i=1}^d \mu_i |u_i|_4^4+\sum_{j\neq i} b|u_i u_j|^2_2}=1
\] 
and therefore
\[
I_d(tW)=\frac{t^2}{4}\sum_{i=1}^d \|w_i\|_{\lambda_i}^2 \leq \frac{t^2}{4}\sum_{i=1}^d \|u_i\|_{\lambda_i}^2\leq I_d(\mathbf{u})=c.
\] 
This implies that $tW$ is a ground state. Since all the above inequalities are, in fact, equalities, $t=1$ and actually $W$ is a ground state. Moreover,
\[
f\left(\frac{|u_1|}{u},\ldots, \frac{|u_k|}{u}\right)=f(X_0) \qquad \text{ for a.e. } x\in \R^N,
\]
so that, if we write $\mathbf{u}(x)=( X(x)u(x), u_k(x),\dots, u_d(x)),$ with $X_i=|u_i|/u$, then $X\in \mathcal{X}$ for a.e. $x\in \R^N$. Observe that $X\in C^\infty$ as $\mathbf{u}$ and $u$ are both smooth and $u\neq 0$.

Let us now check that $(u,u_{k+1},\ldots, u_d)\in G_{d-k+1}$. Since $X_0$ is a solution to the maximization problem \eqref{maximizacao}, there exists a Lagrange multiplier $\mu\in \er$ such that
\begin{equation}\label{Xlagrange}
\mu (X_0)_i = \mu_i (X_0)_i^3 + b\sum_{j=1, j\neq i}^{k} (X_0)_j^2(X_0)_i,\ i=1,\dots,k.
\end{equation}
Multiplying by $(X_0)_i$ and summing in $i$, we obtain $\mu=f_{max}$. Hence, since $W$ is a solution of \eqref{sistema}, the pair $(u,u_{k+1},\dots, u_d)$ must be a solution of
\begin{equation}\label{sistemareduzido}
\left\{\begin{array}{l}
\displaystyle -\Delta u + \lambda u= \mu u^3 + b u \sum_{j>k}u_j^2\\[10pt]
 -\Delta u_i + \lambda_i u_i = \mu_iu_i^3 +  bu^2u_i+ b u_i \mathop{\sum_{j>k}}_{j\neq i}bu_j^2,\ i=k+1,\dots,d.
\end{array}\right.
\end{equation}
Moreover, the minimality of $W$ implies that $(u,u_{k+1},\dots,u_d)\in G_{d-k+1}$.

Finally, let us see that actually $X$ is constant. We know that the vector $\mathbf{u}=(Xu,u_{k+1},\dots,u_d)$ is also a solution of \eqref{sistema}. Inserting this expression onto system \eqref{sistema} and using the above equations, we see that
\begin{equation}
u \Delta X_i + 2 \nabla X_i\cdot \nabla u =0, \ i=1,\dots, k.
\end{equation}
It is now a simple exercise to check that this implies that
$$
\int |\nabla X_i|^2|u|^2=0,\quad i=1,\dots,k.
$$
Since $u>0$, we conclude that $X$ is constant. Therefore
\begin{equation}
G\subseteq \{(Xu,u_{k+1},\dots, u_d)\in (H^1(\er^N))^d: (u,u_{k+1},\dots, u_d)\in G_{d-k+1}, X\in \mathcal{X}\}.
\end{equation}
The other inclusion comes from the fact that, given $X_0\in \mathcal{X}$, $\mu=f_{max}$ and $(u,u_{k+1},\ldots, u_{d})\in G_{d-k+1}$, then $(X_0u,u_{k+1},\ldots, u_d)$ is a ground state solution of \eqref{sistema}.
\medbreak

\noindent\textit{Step 2. Expression of $\mathcal{X}$.} Take $X=(x_1,\dots,x_{k})$ such that $|X|=1$. Then
$$
f(X)=\sum_{i=1}^{k} \mu_i x_i^4 + \sum_{i=1}^k bx_i^2(1-x_i^2)= b + \sum_{i=1}^{k}x_i^4(\mu_i-b).
$$
Define $g(z_1,\cdots, z_{k})=b + \sum_{i=1}^{k}z_i^2(\mu_i-b)$. Then $X$ is a maximizer of $f$ on the unit ball iff $Z=(x_1^2,\cdots, x_{k}^2)$ is a maximizer of $g$ on the convex set $\Delta_{k}=\{Z\in\er^{k}: z_i\geq 0,\ \sum_i z_i=1\}$. Now one must split in several cases:
\begin{itemize}
\item If $\max_i \mu_i>b$, writing $\Delta_{k}^+=\{Z\in\er^{k}: \sum_i z_i=1; z_i=0,\forall i:\mu_i\le b\}$,
$$
\max_{Z\in \Delta_{k}} g(Z)<  \max_{Z\in \Delta_{k}^+}g(Z)
$$
Since $g$ is a strictly convex function on $\Delta_{k}^+$, its maximum is attained at some vertex, hence the maximizers are $e_i$, for $i$'s such that $\mu_i=\max\{\mu_1,\ldots, \mu_k\}$.
\item If $\max_i \mu_i<b$, then $g$ is a strictly concave function and hence the maximum is attained at a unique point $Y=(y_1,\cdots,y_{k})$ on the interior of $\Delta_{k}$. This implies that, for some Lagrange multiplier $\eta\in \er$,
$$
y_i(\mu_i-b)=\eta \qquad \textrm{ for every $i=1,\ldots, k$.}
$$
Multiplying the $i$-th equation by $y_i$ and summing up, we obtain $\eta=f_{max}-b$ and therefore
$$
y_i=\frac{f_{max}-b}{\mu_i-b};
$$
\item Finally, if $\max_i \mu_i=b$, writing $\Delta_{k}^0=\{Z\in\er^{k}: \sum_i z_i=1; z_i=0,\forall i:\mu_i< b\}$,
$$
\max_{Z\in \Delta_{k}} g(Z)\le  \max_{Z\in \Delta_{k}^0}g(Z).
$$
Since $g$ is constant on $\Delta_{d-1}^0$, we obtain the desired expression.\qedhere
\end{itemize}
\end{proof}

%We now proceed to determine the elements of $\mathcal{X}$. Take $X=(x_2,...,x_d)\in\mathcal{X}$. From \eqref{Xlagrange}, we see that $Y=(y_1,..,y_{d-1}), \ y_i=(x_{i+1}/f_{max})^2$ satisfies
%\begin{equation}
%MY:=\left[ \begin{array}{cccc}
%\mu_2 & b & \cdots & b\\
%b & \mu_3 & \cdots & b\\
%\vdots &\vdots & \ddots & b\\
%b & b & \cdots & \mu_d
%\end{array}\right]Y=\left[\begin{array}{c}
%1 \\ 1\\ \vdots \\ 1
%\end{array}\right]
%\end{equation}
%We shall consider the case $b=1$, since any other case may be easily reduced to this one.
%Define, for any $k\in\en$ and any ensemble $\{\nu_i\}_{1\le i\le k}$,
%$$
%A(\nu_1,...,\nu_k)=\left|\begin{array}{ccccc}
%1 & 1 & 1 & \cdots & 1\\
%1 & \nu_1 & 1 & \cdots & 1\\
%1 & 1 & \nu_2 & \cdots & 1\\
%\vdots & \vdots & \vdots & \ddots & 1\\
%1 & 1 & 1 & 1 & \nu_k
%\end{array}\right|, \ D(\nu_1,...,\nu_k)=\left|\begin{array}{cccc}
%\nu_1 & 1 & \cdots & 1\\
%1 & \nu_2 & \cdots & 1\\
% \vdots & \vdots & \ddots & 1\\
%1 & 1 & 1 & \nu_k
%\end{array}\right|.
%$$
%In the following, the symbol $\ \hat{}\ $ represents omission. One easily deduces the following formula:
%\begin{equation}
%A(\nu_1,\cdots,\nu_k)=D(\nu_1,\cdots,\nu_k) - \sum_{i=1}^k A(\nu_1,\cdots,\hat{\nu}_{i}, \cdots, \nu_k)
%\end{equation}
%On the other hand, it is elementary to show that
%$$
%A(\nu_1,\cdots,\nu_k)=\prod_{i=1}^{k} (\nu_i - 1).
%$$
%One has $ D(\nu_1,\cdots,\nu_k)=0 $ iff
%$$
%\prod_{i=1}^{k} (\nu_i-1) + \sum_{j=1}^{k}\prod_{i\neq j} (\nu_i -1)=0.
%$$ 

Now we already have the required tools to prove Existence Result I.

\begin{proof}[Proof of Theorem \ref{Tadmissibilidade2}] Along this proof, we denote by $c(\lambda_1,\lambda_2,\dots,\lambda_d)$ the ground state action level of system \eqref{sistema}, and  by $c^{sem}(\lambda_1,\lambda_2,\dots,\lambda_d)$ the semitrivial ground state level, that is
\[
c^{sem}(\lambda_1,\lambda_2,\dots,\lambda_d)=\min\{c(I):\ I\subset\{1,\ldots, d\}, \ \#I\leq d-1\}.
\]
Our aim is to prove that
\begin{equation}\label{eq:finalgoal}
c(\lambda_1,\lambda_2,\dots,\lambda_d)<c^{sem}(\lambda_1,\lambda_2,\dots,\lambda_d).
\end{equation}
Since the proof is long, we divide it into several steps.

\medbreak

\noindent \emph{Step 1.}  It is enough to prove \eqref{eq:finalgoal} in the case $\mu_i=0$ for all $i$ and $b=1$. In fact, by considering the scaling $U_i=\sqrt{b}u_i$ and by the continuity of the levels $c$ and $c^{sem}$, one may consider the case $b$ large as a perturbation of this case  (a complete justification of this procedure is made in Proposition 11 of \cite{Correia3}). 
\medbreak

\noindent \emph{Step 2.} Also, using the scaling $U_i(x)=\lambda_1^{-1/2}u_i(\lambda_1^{-1/2} x)$, the vector $(\lambda_1, \lambda_2, \ldots,\lambda_d)$ becomes $(1,\lambda_2/\lambda_1,\ldots,\lambda_d/\lambda_1)$. Hence, one may focus on the case $\lambda_1=1$, to simplify the notations.

\medbreak

Let us now start with the core of the proof. From Lemma \ref{lemma:monotonicity_of_gsl}, one has the following properties:
\begin{equation}\label{eq:relacao1}
c(1,\lambda_2,\dots,\lambda_d)\le c(1,\lambda_d,\dots,\lambda_d)
\end{equation}
\begin{equation}\label{eq:relacao2}
c^{sem}(1,\lambda_2,\dots,\lambda_d)\ge c^{sem}(1,\lambda_2,\dots,\lambda_2)
\end{equation}
and, using a suitable scaling,
\begin{align}\label{eq:relacao3}
c(1,\lambda_d,\dots,\lambda_d)&=\left(\frac{\lambda_d}{\lambda_2}\right)^{\frac{4-N}{2}} c\left(\frac{\lambda_2}{\lambda_d},\lambda_2,\dots,\lambda_2\right)\\&\le \left(\frac{\lambda_d}{\lambda_2}\right)^{\frac{4-N}{2}}c(1,\lambda_2,\dots,\lambda_2).
\end{align}
We will now focus our attention on $c(1,\lambda_2,\dots,\lambda_2)$, proving that
\begin{equation}\label{eq:nextgoal}
c(1,\lambda_2,\dots,\lambda_2)<K c^{sem}(1,\lambda_2,\ldots, \lambda_2),
\end{equation}
for some constant $K>0$.
\medbreak

\noindent \emph{Step 3.} Now take a semitrivial $\mathbf{u}=(u_1,\dots, u_d)$ that achieves $c^{sem}(1,\lambda_2,\dots,\lambda_2)$. Since $\lambda_1=\min \{\lambda_1,\ldots, \lambda_d\}$, using an argument similar to that in the proof of theorem \ref{Tadmissibilidade1}, one has $u_1\neq 0$. Since one may order the remaining components as one wishes, we shall suppose that $u_d=0$. Next, we want to apply Theorem \ref{Tlambdaigual}. To that end, following the notations of that theorem, we determine explicitly $f_{max}$ and $\mathcal{X}$. First of all, since $\mu_i=0$, one has
$$
\mathcal{X}=\left\{X=(x_2,\dots,x_{d-1})\in \er^{d-2}: x_i=\pm \left(1-f_{max}\right)^{1/2}\right\}.
$$
On the other hand, since $\mathcal{X}$ is the set of solutions of the maximization problem
\begin{equation}
f(X_0)=f_{max}=\max_{|X|=1} f(X), \quad |X_0|=1,
\end{equation}
any element in $\mathcal{X}$ has unit norm. Since $$\left(\left(1-f_{max}\right)^{1/2},\dots, \left(1-f_{max}\right)^{1/2}\right)\in\mathcal{X},$$
we obtain $\left(1-f_{max}\right)^{1/2}=(d-2)^{-1/2}$, that is, $f_{max}=1-1/(d-2)$. Theorem \ref{Tlambdaigual} now implies that
$$
\mathbf{u}=(u_1,\pm (d-2)^{-1/2}u,\dots, \pm (d-2)^{-1/2}u,0)
$$
where $(u_1,u)$ is a (nontrivial) ground-state of
\begin{equation}
\label{sistwred}
\left\{\begin{array}{l}
-\Delta w_1 + w_1 =  w_2^2w_1\\
-\Delta w_2 + \lambda_2  w_2= \mu w_2^3 + w_1^2w_2
\end{array}\right., \quad \mu=1-\frac{1}{d-2}.
\end{equation}
\medbreak

\noindent \emph{Step 4.} We will provide an estimate of the $L^4$ norm of $u_1$ in terms of the $L^4$ norm of $u$. To this end, consider 
$$\mathbf{w}=(w_1,w_2)=\Big(u_1,\frac 1{\sqrt{\lambda_2}}u_1\Big).$$
Since $\lambda_2\geq 1$, a simple computation yields
\begin{equation}
 \label{filipe1}
\|w_1\|_1^2+\|w_2\|_{\lambda_2}^2=\|u_1\|_1^2+\frac 1{\lambda_2}\|u_1\|_{\lambda_2}^2\leq 2\|u_1\|_1^2.
\end{equation}
Furthermore, since $(u_1,u)$ is in particular a solution of \eqref{sistwred}, we obtain that
$$\|u_1\|_1^2=\int u_1^2u^2=\|u\|_{\lambda_2}^2-\mu \int u^4\leq \|u\|_{\lambda_2}^2.$$
Combining this with \eqref{filipe1}, we get $\|w_1\|_1^2+\|w_2\|_{\lambda_2}^2\leq \|u_1\|_1^2+\|u\|_{\lambda_2}^2$. %Hence, 
%putting $J(v_1,v_2)= \mu\int v_2^4+2\int v_1^2v_2^2$,
Now
it is easy to see that 
%$J(w_1,w_2)\leq J(u_1,u)$, that is
\begin{equation}
 \label{filipe2}
\mu\int w_2^4+2\int w_1^2w_2^2=\frac {\mu+2\lambda_2}{\lambda_2^2}\int u_1^4\int u_1^4\leq \mu\int u^4+2\int u_1^2u^2.
\end{equation}
Indeed, if the converse inequality was true, by choosing $t$ such that $t{\bf w}=(tw_1,tw_2)$ lies in the Nehari manifold associated to \eqref{sistwred}, one would get that $t<1$, and the energy level associated to $t\bf{w}$ is strictly less than the one associated to $(u_1,u)$, which is absurd since the latter vector is a ground-state.

Now, from \eqref{filipe2}, for any $\varepsilon>0$,
$$\frac {\mu+2\lambda_2}{\lambda_2^2}\int u_1^4\leq \mu\int u^4+\varepsilon\int u_1^4+\frac 1{\varepsilon}\int u^4.$$
Choosing $\displaystyle \varepsilon=\frac {\mu+2\lambda_2}{2\lambda_2^2}$ one obtains the estimate
\begin{equation}
 \label{filipe3}
 \int u_1^4\leq 2\lambda_2^2\frac{(\mu+\lambda_2)^2+\lambda_2^2}{(\mu+2\lambda_2)^2}\int u^4,
\end{equation}
as wanted.
\medbreak

\noindent \emph{Step 5.} We are now ready to prove \eqref{eq:nextgoal}. Consider 
\[
\mathbf{v}=(v_1,\dots,v_d)=(u_1,(d-1)^{-1/2}u,\dots,(d-1)^{-1/2}u).
\] 
Then 
$$
\|u_1\|_1^2 + \sum_{i=2}^{d} \|u_i\|_{\lambda_2}^2 = \|v_1\|_1^2+ \sum_{i=2}^{d} \|v_i\|_{\lambda_2}^2
$$ 
and
\begin{align*}
&\sum_{{\substack{i,j=1\\  j\neq i}}}^{d}|v_iv_j|_2^2-\sum_{{\substack{i,j=1\\  j\neq i}}}^{d}|u_iu_j|_2^2= 
\left(2\sum_{\substack{j=2}}^{d}|v_1v_j|_2^2 - 2\sum_{\substack{j=2}}^{d}|u_1u_j|_2^2\right) \\&+ \left(\sum_{\substack{i,j=2\\  j\neq i}}^{d}|v_iv_j|_2^2-\sum_{{\substack{i,j=2\\  j\neq i}}}^{d}|u_iu_j|_2^2\right) 
=\left(\frac{d-2}{d-1}-\frac{d-3}{d-2}\right)|u|_4^4>0.
\end{align*}
If one defines
$$
t^2=\frac{\displaystyle\sum_{i=1}^{d} \|v_i\|_{\lambda_i}^2}{\displaystyle \sum_{{\substack{i,j=1\\  j\neq i}}}^{d}|v_iv_j|_2^2}
$$
then $t\mathbf{v}\in\mathcal{N}(1, \lambda_2, \dots, \lambda_2)$ (the Nehari manifold associated with the ground-state level $c(1,\lambda_2,\dots, \lambda_2)$) and
\begin{align*}
t^2&=\frac{\displaystyle\sum_{i=1}^{d} \|v_i\|_{\lambda_i}^2}{\displaystyle \sum_{{\substack{i,j=1\\  j\neq i}}}^{d}|v_iv_j|_2^2} = \frac{\displaystyle \sum_{i=1}^{d} \|u_i\|_{\lambda_i}^2}{\displaystyle \left(\frac{d-2}{d-1}-\frac{d-3}{d-2}\right)|u|_4^4 + \sum_{{\substack{i,j=1\\  j\neq i}}}^{d}|u_iu_j|_2^2}
\end{align*}
\begin{align*}
&=1-\frac{\displaystyle\left(\frac{d-2}{d-1}-\frac{d-3}{d-2}\right)|u|_4^4}{\displaystyle \left(\frac{d-2}{d-1}-\frac{d-3}{d-2}\right)|u|_4^4 + \sum_{{\substack{i,j=1\\  j\neq i}}}^{d}|u_iu_j|_2^2}\\
&=1-\frac{\displaystyle\frac{d-2}{d-1}-\frac{d-3}{d-2}}{\frac{|u_1u|_2^2}{|u|_4^4} + \frac{d-2}{d-1}}\le 1- \frac{\displaystyle\frac{d-2}{d-1}-\frac{d-3}{d-2}}{\frac{|u_1|_4^2}{|u|_4^2} + \frac{d-2}{d-1}}\\
&\le 1- \frac{\displaystyle\frac{d-2}{d-1}-\frac{d-3}{d-2}}{\displaystyle\sqrt{2\lambda_2^2\frac{(\mu+\lambda_2)^2+\lambda_2^2}{(\mu+2\lambda_2)^2}} + \frac{d-2}{d-1}} =: C(\lambda_2,d)^2
\end{align*}
Hence
\begin{equation}\label{eq:finalrelation}
c(1,\lambda_2,\dots,\lambda_2)\le I_d(t\mathbf{v})=t^2I_d(\mathbf{u})=t^2c^{sem}(1,\lambda_2,\dots,\lambda_2).
\end{equation}
\medbreak

\noindent \emph{Step 6.} Putting together all the relations between the various action levels, namely \eqref{eq:relacao1}, \eqref{eq:relacao2}, \eqref{eq:relacao3} and  \eqref{eq:finalrelation}, one arrives to
\begin{align*}
c(1,\lambda_2,\dots,\lambda_d)&\le t^2\left(\frac{\lambda_d}{\lambda_2}\right)^{\frac{4-N}{2}} c^{sem}(1,\lambda_2,\dots,\lambda_d)\\&\le C^2\left(\frac{\lambda_d}{\lambda_2}\right)^{\frac{4-N}{2}} c^{sem}(1,\lambda_2,\dots,\lambda_d)<c^{sem}(1,\lambda_2,\dots,\lambda_d),
\end{align*}
provided that $\lambda_d<\alpha\lambda_2$, with $\alpha=C^{-\frac{4}{4-N}}$. \qedhere
\end{proof}

Finally, we close this section with the proof of the existence result for the situation where the interaction coefficients $b_{ij}$ do not necessarily coincide.

\begin{proof}[Proof of Theorem \ref{existbetageral}]
Once again, we will proceed by mathematical induction on the number of equations. Following the steps of the proof of Theorem \ref{Tadmissibilidade1}, in order to prove this result we only need to exhibit $w\in H^1(\er^N)$ such that
\begin{equation}
\label{sete}
 \|w\|_{\lambda}^2<\sum_{i=1}^{d-1}b_{id}\int u_i^2w^2
\end{equation}
which is the analogous of condition \eqref{dois} for the system at hand. Here, the vector $(u_1,\cdots u_{d-1})$, assumed by induction hypothesis to be fully nontrivial, is a ground state achieving the level $c^{sem}=c(\{1,2,\dots,d-1\})$ with $c^{sem}\leq c(I)$ for any $\# I\leq d-1$.\\  
Without loss of generality, we may assume once again that for all $1\leq i\leq d-1$, $|u_1|_4\geq |u_i|_4$. Multiplying 
$$ -\Delta u_1+\lambda u_1=\mu_1|u_1|^2u_1+\sum_{j=2}^{d-1}b_{1j}u_j^2u_1$$
by $u_1$ and integrating by parts leads to
$$
\|u_1\|_{\lambda}^2=\mu_1 |u_1|_4^4+\sum_{j=2}^{d-1}b_{1j}\int u_j^2u_1^2.
$$
Hence, taking $w=u_1$, condition \eqref{sete} is equivalent to
$$\mu_1 |u_1|_4^4+\sum_{j=2}^{d-1}b_{1j}\int u_j^2u_1^2<\sum_{i=1}^{d-1}b_{id}\int u_i^2u_1^2,$$
that is
\begin{equation}
 \label{cw2}
\sum_{j=2}^{d-1}(b_{jd}-b_{1j})\int u_j^2u_1^2+(b_{1d}-\mu_1)\int u_1^4>0.
\end{equation}
 Therefore,
$$\int u_j^2u_1^2<\frac 12\Big(\int u_j^4+\int u_1^4\Big)<\int u_1^4$$
and 
$$\int u_1^4>\frac 1{d-2}\int \sum_{j=2}^{d-1}u_j^2u_1^2.$$
Finally, for \eqref{sete} to hold, it is sufficient that
\begin{equation}
\sum_{j=2}^{d-1}\Big(b_{jd}-b_{1j}+\frac{b_{1d}-\mu_1}{d-2}\Big)\int u_j^2u_1^2>0.
\end{equation}
and, in particular, the conditions stated in Theorem \ref{existbetageral}.\qedhere
\end{proof}

\begin{section}{Nonexistence of nontrivial ground states}\label{sec:Nonexistence}

In this section we prove the main nonexistence results, namely Theorems \ref{Tnaoexistencia1}, \ref{Tbetanaoadmissivel2} and \ref{Tnaoexistencia2} of the Introduction.

\begin{proof}[Proof of Theorem \ref{Tnaoexistencia1}]
Along this proof, we will denote by $\mathbf{u}_{b}=(u_{1,b},\ldots, u_{d,b})$ a solution of the system \eqref{sistema}, highlighting the dependence on $b$.  First of all, without loss of generality, we can assume that $\lambda_2=1$ (otherwise, exactly as in Step 2 of the proof of Theorem \ref{Tadmissibilidade2}, after a scaling we can pass from $(\lambda_1,\lambda_2,\ldots, \lambda_d)$ to $(\lambda_1/\lambda_2,1,\lambda_3/\lambda_2,\ldots, \lambda_d/\lambda_2))$.

We divide  the proof in three steps.
\medbreak

\noindent \emph{Step 1.} Define $U_{i,b}(x):=\sqrt{b}\, u_{i,b}(x)$, which solves
\begin{equation}\label{eq:normalized_system}
-\Delta U_{i,b}+\lambda_i U_{i,b}=\frac{\mu_i}{b}U_{i,b}^3+U_{i,b} \sum_{j\neq i} U_{j,b}^2.
\end{equation}
We claim that there exists $C=C(\lambda_1)$ such that $\|U_{i,b}\|_{\lambda_i}^2\leq C$.
Associated to system \eqref{eq:normalized_system}, we denote the corresponding action function by $I_{\lambda,\frac{\mu}{b}}$, the Nehari manifold by $\mathcal{N}_{\lambda,\frac{\mu}{b}}$, and the ground state level by $c({\lambda,\frac{\mu}{b}})$. Observe that
\[
\frac{1}{4}\sum_{i=1}^d \|U_{i,b}\|_{\lambda_i}^2 =c\left({\lambda,\frac{\mu}{b}}\right) \leq c({\lambda,0})
\]
by Lemma \ref{lemma:monotonicity_of_gsl}. On the other hand,
\begin{align*}
c({\lambda,0})&= \inf\{ I_{\lambda,0}(\mathbf{u}):\ \mathbf{u}\neq 0,\ \mathbf{u}\in \mathcal{N}_{\lambda,0}\}\\
			&\leq \inf\{ I_{\lambda,0}(u_1,u_2,0,\ldots,0):\ (u_1,u_2)\neq (0,0),\ (u_1,u_2,0,\ldots, 0)\in \mathcal{N}_{\lambda,0}\}.
\end{align*}
Then $\|U_{i,b}\|_{\lambda_i}^2\leq 4c$, where $c$ is the ground state level of the two equations system
\[
\begin{cases}
-\Delta w_1+\lambda_1 w_1=w_1w_2^2\\
-\Delta w_2+ w_2=w_2w_1^2.
\end{cases}
\]
Thus the claim made above follows.

\medbreak

\noindent \emph{Step 2.} Next, we prove that there exists $C=C(\lambda_1)$ such that
\begin{equation}\label{eq:Brezis-Kato}
|\sqrt{b}\, u_{i,b}|_\infty =|U_{i,b}|_\infty\leq C,\qquad \forall b>\max\{\mu_1,\ldots, \mu_d\}.
\end{equation}
Observe that, since $N\leq 3$, we have the continuous embedding $H^1(\R^N)\hookrightarrow L^6(\R^N)$. 

Let us perform a standard Brezis-Kato type argument to pass from $H^1$ to $L^\infty$ bounds. First of all, if there is a sequence $p_k\to \infty$ such that $|U_{i,b}|_{p_k}\leq 1$, the conclusion is obvious for $C:=1$. Suppose that $U_{i,b}\in L^{2+2\delta}(\R^N)$ for some $\delta>0$. We test the equation for $U_{i,b}$ in \eqref{eq:normalized_system} with $U_{i,b}^{1+\delta}$, obtaining
\begin{multline*}
\frac{1+\delta}{(1+\delta/2)^2} \int \left( |\nabla U_{i,b}^{1+\delta/2}|^2+\lambda_i |U_{i,b}|^{2+\delta}\right)=\frac{\mu_i}{b}\int |U_{i,b}|^{4+\delta}\\+\int |U_{i,b}|^{2+\delta} \sum_{j\neq i} U_{j,b}^2.
\end{multline*}
Since $\lambda_1\leq \lambda_i$ and $\mu_i/b\leq 1$, we deduce that
\begin{multline*}
\min\left\{\frac{1+\delta}{(1+\delta/2)^2},\lambda_1\right\} \int \left( |\nabla |U_{i,b}|^{1+\delta/2}|^2+\lambda_i |U_i,b|^{2+\delta}\right) \\
\leq \int |U_{i,b}|^{1+\delta}(U_{i,b}^3+U_{i,b}\sum_{j\neq i}U_{j,b}^2) \leq |U_{i,b}|_{2+2\delta}^{1+\delta}|h_b|_2
\end{multline*}
where $h_b:=U_{i,b}^3+U_{i,b}\sum_{j\neq i}U_{j,b}^2$. Thus
\[
C_6^2 \min\left\{ \frac{1+\delta}{(1+\delta/2)^2},\lambda_1 \right\}  |\, U_{i,b}^{1+\delta/2}\,|_6^2 \leq  |U_{i,b} |_{2+2\delta}^{1+\delta} |h_b|_2
\]
or, equivalently, 
\[
|U_{i,b}|_{6+3\delta}\leq \left(\frac{1}{C_6^2} \max\left\{ \frac{(1+\delta/2)^2}{1+\delta},\frac{1}{\lambda_1}   \right\} |h_b|_2\right)^\frac{1}{2+\delta}|U_{i,b}|_{2+2\delta}^{(1+\delta)/(2+\delta)}.
\]
Assuming, without loss of generality, that $|U_{i,b}|_{2+2\delta}\geq 1$, then from the fact that $(1+\delta)/(2+\delta)\leq 1$ we deduce that $|U_{i,b}|_{2+2\delta}^{(1+\delta)/(2+\delta)}\leq |U_{i,b}|_{2+2\delta}$. Moreover, from the $H^1(\R^N)$ bound of Step 1, there exists $C=C(\lambda_1)$ such that $|h_b|_2\leq C$ for every $b>\max\{\mu_1,\dots,\mu_d\}$. Thus we conclude the existence of $\kappa=\kappa(\lambda_1)$ such that
\[
|U_{i,b}|_{6+3\delta}\leq \left(\kappa \max\left\{ \frac{(1+\delta/2)^2}{1+\delta},\frac{1}{\lambda_1}   \right\} \right)^\frac{1}{2+\delta}|U_{i,b}|_{2+2\delta}.
\]
Now we iterate, by letting
\[
\delta(1)=0,\qquad 2+2\delta(k+1)=6+3\delta(k).
\]
Observe that $\delta(k)\to \infty$, since $\delta(k)\geq (3/2)^{k-2}, k\ge 2$. With this choice of $\delta$ in the previous estimate, we obtain the iterative relation
\[
|U_{i,b}|_{L^{2+2\delta(k+1)}} \le \left( \kappa \max\left\{\frac{(1+\delta / 2)^2}{1+\delta}, \lambda_i \right\} \right)^\frac{1}{2+\delta}  |U_{i,b}|_{2+2\delta(k)}
\]
which, together with $\delta(1)=0$, gives
\begin{eqnarray*}
|U_{i,b}|_{3+6\delta(k)} &=|U_{i,b}|_{2+2\delta(k+1)} \leq \displaystyle\prod_{j=1}^{k}\left[ \kappa \max\left\{\frac{(1+\delta / 2)^2}{1+\delta}, \lambda_i \right\} \right]^{\frac{1}{2+\delta(j)}} |U_{i,b}|_{2}\\
			&\leq \exp\left(\displaystyle \sum_{j=1}^\infty \frac{1}{2+\delta(j)}\log\left[ \kappa \max\left\{\frac{(1+\delta(j) / 2)^2}{1+\delta(j)}, \lambda_i \right\} \right]   \right)|U_{i,b}|_2
\end{eqnarray*}
As $\delta(j)\geq (3/2)^{j-2}, j\ge 2$, we see that
\[
\sum_{j=1}^\infty \frac{1}{2+\delta(j)}\log\left[ C \frac{\left(1+\delta(j)/2\right)^2}{1+\delta(j)} \right] < \infty,
\]
which provides the uniform bound in $L^\infty(\Omega)$

\medbreak

\noindent \emph{Step 3.} Take $i\geq 3$ such that $\lambda_d \geq \ldots \lambda_i >\Lambda:=dC$, where $C$ is the constant appearing in \eqref{eq:Brezis-Kato}. Take $j\in \{i,\ldots, d\}$. By multiplying the equation for $u_{j,b}$ by $u_{j,b}$ itself and recalling that $\mu_j/b\leq 1$ and \eqref{eq:Brezis-Kato} holds, we obtain
\begin{align*}
0\leq \int |\nabla u_{j,b}|^2 &\leq \int \mu_{j,b} u_{j,b}^4+\int b\, u_{j,b}^2 \sum_{k\neq j} u_{k,b}^2-\int \lambda_j u_{j,b}^2\\
					&\leq \int u_{j,b}^2 (dC-\lambda_j)\leq 0,
\end{align*}
which implies that $u_{j,b}\equiv 0$.
\end{proof}

\begin{proof}[Proof of Theorem \ref{Tbetanaoadmissivel2}]
Let $(u_1,\dots,u_d)\in G$. Define $c((\lambda_i;\mu_i;b_{ij})_{(i,j)\in P^2})$ to be the ground state level for the system 
$$
-\Delta w_i + \lambda_i w_i = \mu_i w_i^3 + \sum_{j\in P}b_{ij} w_j^2w_i, \ i\in P.
$$
Then
\begin{align*}
I(u_1,\dots,u_d)&=\min_{(w_1,\dots,w_d)\in\mathcal{N}_d} I(w_1,\dots,w_d) \le \underset{w_i=0,\ i\notin P}{\min_{(w_1,\dots,w_d)\in\mathcal{N}_d}} I(w_1,\dots,w_d)\\&=c((\lambda_i;\mu_i;b_{ij})_{(i,j)\in P^2}).
\end{align*}
Let $b=\min_{(i,j)\in P^2} b_{ij}$. From Lemma \ref{lemma:monotonicity_of_gsl} and a simple normalization argument,

$$c((\lambda_i;\mu_i;b_{ij})_{(i,j)\in P^2})\le c(\lambda_i;0,\dots,0;b,\dots, b)=b^{-1}c(\lambda_i;0,\dots,0;1,\dots,1).$$ 
Hence
$$
I(u_1,\dots,u_d)\le b^{-1}c(\lambda_i;0,\dots,0;1,\dots,1).
$$
Fix $i\notin P$ and suppose that $u_i\neq 0$. Then, for a constant $C>0$, independent of $b_{ij}$,
$$
\|u_i\|_{\lambda_i}^2 = \mu_i |u_i|_4^4 + \int u_i^2\mathop{\sum_{j=1}^k}_{j\neq i} b_{ij}u_j^2\le C \left(\mu_i\|u_i\|_{\lambda_i}^4 + \|u_i\|_{\lambda_i}^2\mathop{\sum_{j=1}^k}_{j\neq i}b_{ij}\|u_j\|_{\lambda_j}^2\right).
$$
Therefore
\begin{align*}
1&\le C\left(\mu_i \|u_i\|_{\lambda_i}^2 + \mathop{\sum_{j=1}^k}_{j\neq i}b_{ij}\|u_j\|_{\lambda_j}^2\right)\le 4C\left(\mu_i + \mathop{\sum_{j=1}^k}_{j\neq i}b_{ij}\right)I(u_1,\dots,u_d)\\
		&\le  4C\left(\mu_i + \mathop{\sum_{j=1}^k}_{j\neq i}b_{ij}\right) b^{-1}c(\lambda_i;0,\dots,0;1,\dots,1),
\end{align*}
which is absurd for $b$ sufficiently large.\end{proof}

\begin{proof}[Proof of Theorem \ref{Tnaoexistencia2}]
Consider a fully nontrivial ground state $(u_1,u_2,\dots,u_d)\in\mathcal{N}_d$. We will check that necessarily $b\geq 2^{1-\frac{d}{2}}\sqrt{\mu_1\mu_d}$.\\
With this intent, let us start by computing $t>0$ such that $(tu_1,0,\dots,0)\in\mathcal{N}_d$:
$$(tu_1,0,\dots,0)\in\mathcal{N}_d\Leftrightarrow t^2\|u_1\|_{\lambda_1}^2=t^4\mu_1|u_1|_4^4\Leftrightarrow t^2=\frac{\|u_1\|_{\lambda_1}^2}{\mu_1|u_1|_4^4}.$$ 
Since $(u_1,u_2,\dots,u_d)$ is a ground state, $I_d(u_1,u_2,\dots,u_d)\leq I_d(tu_1,0,\dots,0)$, that is
$$\frac 14\Big( \sum_{i=1}^d \mu_i|u_i|_{4}^{4}+2b \sum_{i<j}|u_iu_j|_2^2  \Big)\leq \frac 14\mu_1t^4|u_1|_4^4,$$
i.e.,
$$\mu_1|u_1|_4^4\Big(\sum_{i=1}^d \mu_i|u_i|_{4}^{4}+2b \sum_{i<j}|u_iu_j|_2^2\Big)\leq \|u_1\|_{\lambda_1}^4.$$
Multiplying the first line of system (\ref{sistema}) by $u_1$ and integrating then yields
$$\mu_1|u_1|_4^4\Big(\sum_{i=1}^d \mu_i|u_i|_{4}^{4}+2b \sum_{i<j}|u_iu_j|_2^2\Big)\leq \Big(\mu_1|u_1|_4^4+b\sum_{j=2}^d |u_1u_j|_2^2\Big)^2.$$
Hence
\begin{align*}
\mu_1|u_1|_4^4\sum_{i=2}^d\mu_i|u_i|_4^4+2b\mu_1|u_1|_4^4\sum_{i<j}|u_iu_j|_2^2&\leq 2b\mu_1|u_1|_4^4\sum_{j=2}^d |u_1u_j|_2^2\\&+b^2\Big(\sum_{j=2}^d |u_1u_j|_2^2\Big)^2,
\end{align*}
that is 
\begin{align*}
\mu_1|u_1|_4^4\sum_{i=2}^d\mu_i|u_i|_4^4+2b\mu_1|u_1|_4^4\sum_{1<i<j}|u_iu_j|_2^2&\leq b^2\Big(\sum_{j=2}^d |u_1u_j|_2^2\Big)^2\\&\leq  b^2\Big(\sum_{j=2}^d |u_1|_4^2|u_j|_4^2\Big)^2.
\end{align*}
Finally, 
$$\sum_{i=2}^d \mu_1\mu_i|u_i|_4^4\leq b^2\Big(\sum_{j=2}^d |u_j|_4^2\Big)^2\leq b^22^{d-2}\sum_{j=2}^d |u_j|_4^4.$$
From this inequality, we obtain that
$$\displaystyle b^22^{d-2}\geq \mu_1\frac{\displaystyle\sum_{i=2}^d \mu_i|u_i|_4^4}{\displaystyle\sum_{i=2}^d |u_i|_4^4}\geq \mu_1\min\{\mu_i\,:\,i\neq 1\}.$$ 
By interchanging the roles of $\mu_1$ and $\mu_j$, $j\geq 2$, we get that for all $j\in\{1,\dots,d\}$,$$b^22^{d-2}\geq \mu_j\min\{\mu_i\,:\,i\neq j\}.$$
In particular,
$$b\geq \frac{\sqrt{{\mu_1}\mu_d}}{2^{\frac d2-1}}.$$
\end{proof}
\end{section}

\section{Open Problems}\label{sec:OpenProblems}

Having established qualitatively under which conditions the ground states are either fully nontrivial or semitrivial, the next natural question would be to obtain the exact thresholds for this transition. As we referred, in the 2-equation case this was done by Mandel \cite{Mandel}. In such a case, the situation seems simpler, since one can use the well-known fact that the problem $-\Delta w+w=w^3$ has a unique  positive solution, up to translation (see \cite{Kwong}), and since there is a unique interaction parameter, $b_{12}=b_{21}$. 
\medbreak

Another challenging question would be, in the situations where we proved that the ground states are semitrivial, to determine if there are fully nontrivial solutions (bound states), obviously with higher levels of action. In the general case, only very particular results are known, namely when $b_{ij}$ are small, or for some particular combinations of $b_{ij}$ small and large. \\
For small $b_{ij}$ (weak cooperation) this was first shown in \cite{LinWei} (check Theorem 2) and also in \cite[Corollary 1.3]{SoaveTavares} (which also allows negative coefficients). Check also the papers \cite{LiuLiuChang,Mandel,WangWillem} for other existence results in the weak cooperation setting.\\
For existence results with simultaneous weak and strong cooperation, in \cite[Theorem 1.4\& 1.5]{SoaveTavares} it is proved for instance that there exists a positive solution in the following situation: divide the components in $m$ groups: $\{1,\ldots, d\}=\cup_{h=1}^m I_h$ with $I_h\cap I_k=\emptyset$ for $h\neq k$, and ask that:
\begin{itemize}
\item $\lambda_i\equiv \lambda_h$ for every $i\in I_h$, $h=1,\ldots, m$;
\item $b_{ij}\equiv b_h$ ``large'', for every $i,j\in I_h$ with $i\neq j$, $h=1,\ldots, m$;
\item $b_{ij}$ ``small'', for every $i\in I_h$, $j\in I_k$ with $h\neq k$.
\end{itemize} 
See also \cite[Section 4]{Colorado} for other results with large and small $b_{ij}$, of perturbative nature. \\
In some of these papers, it is shown that the solutions that are found minimize the action among the set of positive solutions. A complete description of the general picture is open.

\appendix

\section{Appendix}
Given $\lambda=(\lambda_1,\ldots, \lambda_d)$, $\mu=(\mu_1,\ldots, \mu_d)$ and $B=(b_{ij})_{i\neq j}$, denote by $c({\lambda,\mu,B})$ the ground state level of the system
\[
-\Delta u_i+\lambda_i u_i= \mu_i u_i^3+b u_i\sum_{j\neq i} u_j^2.
\]

We close this paper with a known fact regarding monotonicity of the ground state levels with respect to the parameters. Since we couldn't find a reference in the litterature covering all the cases we need, for the readers convenience we quickly state and prove the result here.
\begin{Lema}\label{lemma:monotonicity_of_gsl}
Take $\lambda, \tilde \lambda$, $\mu,\tilde \mu$ and $B, \tilde B$ such that $\lambda_i\leq \tilde \lambda_i$, $\tilde \mu_i\leq  \mu_i$ for every $i$, and $0< \tilde b_{ij}\leq b_{ij}$ for every $i\neq j$. Then
\[
c({\lambda,\mu,B})\leq c({\tilde \lambda,\tilde \mu,\tilde B}).
\]
\end{Lema}

\begin{proof}
This has already been observed partially in \cite{Correia3}. Denote by $I_{\lambda,\mu,B}$ and $N_{\lambda,\mu,B}$ the associated action functional and Nehari manifold to the level $c({\lambda,\mu,B})$. Observe that $I_{\lambda,\mu,B}(u)\leq I_{\tilde \lambda,\tilde \mu,\tilde B}(u)$ for every $u\in (H^1(\R^N))^d$. Take $\mathbf{u}^*$ achieving $c({\tilde \lambda,\tilde \mu,\tilde B})$. Then it is straightforward to check that there exists a unique $t^*>0$ such that $t^*\mathbf{u}^*\in \mathcal{N}_{\lambda,\mu,B}$. Then
\begin{align*}
c({\tilde \lambda,\tilde \mu,\tilde B})&=I_{\tilde \lambda,\tilde \mu,\tilde B}(\mathbf{u}^*)=\max_{t>0} I_{\tilde \lambda,\tilde \mu,\tilde B}(t \mathbf{u}^*) \\
				&\geq \max_{t>0} I_{\lambda,\mu,B}(t \mathbf{u}^*)=I_{\lambda,\mu,B}(t^* u^*) \geq \inf_{\mathcal{N}_{\lambda,\mu,B}} I_{\lambda,\mu,B}=c({\lambda,\mu,B}).
\end{align*}
\end{proof}

\noindent \textbf{Acknowledgements.} 
The authors would like to thank Nicola Soave for his comments on a previous version of this paper.\\
Sim\~ao Correia was partially supported by Funda\c c\~ao
para a Ci\^encia e Tecnologia, 
through the grant 
SFRH/BD/96399/2013
and through contract UID/\-MAT/\-04561/\-2013.\\
Filipe Oliveira was partially supported by Funda\c c\~ao
para a Ci\^encia e Tecnologia, through contract UID/MAT/00297/2013.\\
Hugo Tavares was partially supported by Funda\c c\~ao para a Ci\^encia e Tecnologia through the program Investigador FCT and the project PEst-OE/\-EEI/\-LA0009/\-2013, as well as by the ERC Advanced Grant 2013 n.339958 ``Complex Patterns for Strongly Interacting Dynamical Systems - COMPAT''.

%\bibitem{OT} Filipe Oliveira and Hugo Tavares, Ground states for a nonlinear Schr\"odinger system with sublinear coupling terms.
%\end{thebibliography}

\small
\noindent \textsc{Sim\~ao Correia}\\
CMAF-CIO and FCUL \\
\noindent Campo Grande, Edif\'icio C6, Piso 2, 1749-016 Lisboa (Portugal)\\
\verb"sfcorreia@fc.ul.pt"\\

\noindent \textsc{Filipe Oliveira}\\
Mathematics Department, FCT-UNL\\
Centro de Matemática e Aplicações\\
NOVA University of Lisbon\\
Caparica Campus, 2829-516 (Portugal)\\
\verb"fso@fct.unl.pt"\\

\noindent \textsc{Hugo Tavares}\\
CAMGSD, Mathematics Department\\
Instituto Superior T\'ecnico, Universidade de Lisboa\\
Av. Rovisco Pais, 1049-001 Lisboa (Portugal)\\
\verb"htavares@math.ist.utl.pt"

\end{document}